\documentclass[12pt,a4paper]{article}
\usepackage[intlimits]{amsmath}
\usepackage{amsfonts,amssymb,amscd,amsthm}
\usepackage[all]{xy}
\usepackage{url}
\usepackage{cite}
\usepackage[dvips]{graphicx}

\def\Td{\operatorname{Td}}

\def\Mat{\operatorname{Mat}}

\def\tr{\operatorname{tr}}
\def\End{\operatorname{End}}
\def\Hom{\operatorname{Hom}}

\def\Ga{{\Gamma}}
\def\lra{{\longrightarrow}}
\def\pa{{\partial}}
\def\si{{\sigma}}
\def\om{{\omega}}
\def\Om{{\Omega}}
\def\la{{\lambda}}
\def\ga{{\gamma}}
\def\ch{\operatorname{ch}}
\def\ind{\operatorname{ind}}

\newcommand{\BL}{\biggl}
\newcommand{\BR}{\biggr}
\let\rom\textup

\theoremstyle{plain}
\newtheorem{theorem}{Theorem}

\newtheorem{lemma}[theorem]{Lemma}
\newtheorem{proposition}[theorem]{Proposition}
\theoremstyle{definition}
\newtheorem{definition}[theorem]{Definition}
\newtheorem{example}[theorem]{Example}

\theoremstyle{remark}
\newtheorem{remark}[theorem]{Remark}

\title{Noncommutative elliptic theory.  \\
        Examples}
\author{A.~Yu.~Savin, B.~Yu.~Sternin}

\date{}

\begin{document}
\maketitle


\begin{abstract}
We study differential operators, whose coefficients define noncommutative
algebras. As algebra of coefficients, we consider crossed products,
corresponding to action of a discrete group on a smooth manifold. We give index
formulas for Euler, signature and Dirac operators twisted by projections over
the crossed product. Index of Connes operators on the noncommutative torus is
computed.
\end{abstract}

\section*{Introduction}

Noncommutative elliptic theory is the theory of differential operators, whose
coefficients form algebras, which are in general noncommutative. Operators of
this type first appeared in the work of A.~Connes
 \cite{Con4,Con1}, where on the real line the following operators were considered
\begin{equation*}\label{eq-operators-connes22}
    D=\sum_{\alpha+\beta\le m} a_{\alpha\beta}x^\alpha\left(-i\frac
    d{dx}\right)^\beta.
\end{equation*}
Here the coefficients $a_{\alpha\beta}$ belong to the algebra generated by
operators   $U,V$
$$
(Uf)(x)=f(x+1),\qquad (Vf)(x)=e^{-2\pi i x/\theta}f(x),
$$
of unit shift and product by exponential  ($\theta\in (0,1]$ is some fixed
parameter). It is clear that the operators $U$ and $V$ do not commute. For such
operators, the ellipticity condition and index formula were obtained.

Further examples of noncommutative differential operators were constructed on
manifolds with torus action (see~\cite{CoLa2,LaSu1,CoDu1}) in connection with
studying isospectral deformations and other questions of noncommutative
geometry.

Specialists in the field of differential equations also studied noncommutative
differential equations   (e.g., see survey~\cite{AnLe3} and the references
there). Namely, if a  discrete group acts on a manifold, then one can consider
operators of multiplication by functions and also operators of the group
representation (which are called shift operators, or operators of change of
variables) as coefficients of differential operators. This class of operators
is known in the literature under various names: functional-differential
operators, operators with shifts, nonlocal operators, noncommutative operators.
We use the term ``noncommutative'', because from our point of view it is most
important in  elliptic theory  that the coefficients of these operators form
noncommutative algebras. This class of operators includes operators of
A.~Connes as well as operators on toric manifolds as special cases.

Although finiteness theorem (Fredholm property) for noncommutative differential
equations was established relatively long ago~\cite{Ant1}, the index problem,
i.e., the problem of expressing the index of a noncommutative elliptic operator
in terms of its symbol, was open for a long time. In 2008 this problem was
solved in~\cite{NaSaSt17,NaSaSt18}. The aim of the present paper is to apply
this index formula to operators, which appear in geometry, and compute indices
of some specific operators.

Let us briefly describe the contents of the paper.

In the first part of the paper (Sections 1-3) we review elliptic theory for
noncommutative differential operators: definition of operators, formula for the
symbol, finiteness theorem, construction of Chern character of elliptic symbol,
index theorem. The proofs of all these results were published
in~\cite{NaSaSt17}, except Theorem~\ref{th-e}, which is new. Then in Sections 4
and 5 we obtain index formulas for classical geometric operators: Euler,
signature, Dirac operators. In Section 6, we compute the index of operators on
the circle, while Section~7 deals with the index of operators of A.~Connes on
the noncommutative torus.

The  results of this paper were presented at International conferences
``$C^*$-algebras and Elliptic Theory III'' in Banach center (Poland), January
26 - 31, 2009 and ``$K$-theory, $C^*$-algebras and topology of manifolds'' in
Chern Institute of Mathematics (China), June 1-5, 2009. We are also grateful to
Professors A.S.~Mishchenko and C.~Ogle for useful remarks.

Supported by RFBR (08-01-00867), Council of the President of Russian Federation
(NSh-1562.2008.1), Deutsche Forschungsgemeinschaft (436 RUS
113/849/0-1\circledR ``$K$-theory and noncommutative geometry of stratified
manifolds''). The work of the first author was also supported by E.~Balzan
prize in mathematics, received by P.~Deligne in 2004.

\section{Finiteness theorem}

\paragraph{Noncommutative differential operators.}

Let a countable finitely generated group $\Ga$ act on a smooth closed
Riemannian manifold $M$.


Throughout the paper we consider only isometric actions of $\Ga$ and suppose
that the action extends to an action of a compact Lie group $G$. We shall also
assume that the discrete groups under consideration are of polynomial growth.

On $M$ we choose a $G$-invariant Riemannian metric and volume form.

\begin{definition}
\emph{Noncommutative differential operator of order} $\le m$ is an operator of
the form
\begin{equation}\label{eq-nc-operator}
    D=\sum_{g\in \Ga} T(g)D(g):C^\infty(M)\lra C^\infty(M),
\end{equation}
which acts in the space of smooth functions, where
\begin{itemize}
    \item $D(g)$ is a differential operator of order $\le m$,
    \item $T(g)$ is \emph{shift operator} (change of variables)
    \begin{equation}\label{eq-shift-oper}
      (T(g)u)(x)=u(g^{-1}(x)),
    \end{equation}
    corresponding to a diffeomorphism $g:M\to M$.
\end{itemize}
\end{definition}

There is a question of convergence of the series in~\eqref{eq-nc-operator}, if
$\Ga$ is infinite. Therefore, we shall first assume for simplicity that the
operators $D(g)$ are equal to zero, except for finitely many elements $g$. Less
restrictive condition will be given below.

\paragraph{Symbol of noncommutative operator.}

Obviously,  operator \eqref{eq-nc-operator} is determined modulo operators of
order  $\le m-1$ by the set of symbols $\sigma(D(g))$ of operators
$D(g)$.\footnote{Recall that the symbol of a differential operator
$$
P=\sum_{|\alpha|\le m} a_\alpha(x)\left(-i\frac\partial{\partial x}\right)^{\alpha}
$$
of order $\le m$ is a function
$$
 \sigma(P)=\sum_{|\alpha|=m} a_\alpha(x)\xi^{\alpha},\quad (x,\xi)\in T^*M,
$$
defined on the cotangent bundle minus the zero section $T^*M\setminus 0$ and
homogeneous of order $m$.} This set can be naturally considered as a symbol of
operator  $D$ and is denoted by
\begin{equation}\label{eq-symbol}
    \sigma(D)=\{\sigma(D(g))\}_{g\in\Ga}.
\end{equation}
Let us establish the composition formula, i.e., the rule of multiplication of
symbols, corresponding to composition of operators.

The composition of operators $D=\sum_g T(g)D(g)$ and $Q=\sum_h T(h)Q(h)$ is
equal to
\begin{multline}\label{compoz}
 DQ=\sum_{g,h} T(g)D(g)T(h)Q(h)=\sum_{g,h}
 T(gh)\bigl(T(h)^{-1}D(g)T(h)\bigr)Q(h)=\\
 = \sum_k T(k)\bigl(\sum_{gh=k} T(h)^{-1}D(g)T(h)Q(h)\bigr).
\end{multline}
Here the operator $T(h)^{-1}D(g)T(h)$ is a differential operator with symbol
\begin{equation}\label{beta}
\sigma(T(h)^{-1}D(g)T(h))=T({\partial h})\Bigl(\sigma(D(g))\Bigr),
\end{equation}
where $\partial h:T^*M\lra T^*M$ is the codifferential\footnote{The action of
$G$ on $M$ extends to action on the cotangent bundle. Namely, an element $h\in
G$ acts as the codifferential
$$
\partial h=((dh)^t)^{-1},
$$
where $dh:TM\to TM$ is the differential of $h$, and $(dh)^t$ is the dual
mapping of the cotangent bundle.} of diffeomorphism $h$. Note that~\eqref{beta}
is just the usual formula for transformation of symbol under a change of
variables.  So, if we substitute~\eqref{beta} in \eqref{compoz}, we see, that
the product $DQ$ is an operator of the form~\eqref{eq-nc-operator}, and the
symbols of its components are equal to
\begin{equation}\label{eq-composition}
    \sigma(DQ)(k)=\sum_{gh=k} \Bigl(T({\pa h})\sigma(D(g))\Bigr)\sigma(Q(h)).
\end{equation}
Algebras, whose elements are, similar to  \eqref{eq-symbol}, functions on
$\Gamma$ and the product is defined by convolution of the
form~\eqref{eq-composition} are called \emph{crossed products} by $\Gamma$.
Before we give the definition of the crossed product suitable for our purposes,
we give two remarks.

First, we shall consider restrictions of symbols to the cosphere bundle
$S^*M\subset T^*M$, which consists of covectors of unit length. The action of
$G$ restricts to the cosphere bundle, since the action is isometric.

Second, we shall consider classes of functions on $\Gamma$, which are larger
than compactly supported functions, because in the class of functions with
compact support one can not find inverse elements\footnote{Indeed, consider
operator $Id+T(g)/2$, where $g\ne e$. It is invertible. However, the inverse is
given by the Neumann series
$$
(Id+T(g)/2)^{-1}=Id -T(g)/2+T(g^2)/4-\ldots
$$
and can not be written in general as a finite sum, unless $g$ is of finite
order.}.

\begin{definition}\label{def-cross}
\emph{Smooth crossed product} of the Fr\'echet algebra $C^\infty(S^*M)$  and
group $\Ga$ is the algebra of $C^\infty(S^*M)$-valued functions on $\Ga$, for
which the following seminorms
$$
\|f\|_{n,l}=\sup_{g\in\Ga} \Bigl(\|f(g)\|_n(1+|g|)^l\Bigr)
$$
are finite for all $n$ and $l$, where $\{\|\cdot\|_n\}$ runs over some defining
system of seminorms in $C^\infty(S^*M)$, with the product given by
$$
 (f_1f_2)(k)=\sum_{gh=k} (T({\partial h})f_1(g))f_2(h).
$$
\end{definition}
The smooth crossed product is denoted by $C^\infty(S^*M)\rtimes\Ga$.

\begin{remark}
Interested reader can find detailed exposition of smooth crossed products
in~\cite{Schwe1}. In particular, it is proved in the cited paper that the
smooth crossed product is an algebra. Note also that in
Definition~\ref{def-cross} one can replace $C^\infty(S^*M)$ by any Fr\'echet
algebra, on which group $\Ga$ acts.
\end{remark}

\begin{definition}
\emph{Symbol} of noncommutative operator~\eqref{eq-nc-operator} is the
collection~\eqref{eq-symbol}, which is considered as an element of the smooth
crossed product:
$$
\sigma(D)\in C^\infty(S^*M)\rtimes\Ga.
$$
A noncommutative differential operator is called \emph{elliptic}, if its symbol
is invertible in the algebra $C^\infty(S^*M)\rtimes\Ga$.
\end{definition}
The equality~\eqref{eq-composition} means that for noncommutative operators
$D,Q$ one has composition formula
\begin{equation}\label{eq-compos2}
\sigma(DQ)=\sigma(D)\sigma(Q).
\end{equation}

\paragraph{Finiteness theorem.}

\begin{theorem}\label{th-finite}
An elliptic noncommutative operator $D$ is Fredholm as an operator acting in
Sobolev spaces
$$
D:H^s(M)\lra H^{s-m}(M)
$$
for all $s$. Moreover, the kernel and cokernel consist of smooth functions.
\end{theorem}
\begin{proof}[Sketch of the proof.]

1. Using classical theory of pseudodifferential operators   (e.g., see
\cite{KoNi1}), for any $a\in C^\infty(S^*M)\rtimes\Ga$ we can define operator
$A$ of the form \eqref{eq-nc-operator}, which has symbol $a$. In addition,
composition formula~\eqref{eq-compos2} remains valid.

2. Let $Q$ denote an operator with symbol $\sigma(D)^{-1}\in
C^\infty(S^*M)\rtimes\Ga$. Then the operators
$$
 Id-QD\qquad \text{and}\qquad Id-DQ
$$
have zero symbols by composition formula~\eqref{eq-compos2} and, therefore, are
compact operators of order $-1$. Then the desired properties follow from the
standard considerations.
\end{proof}

\paragraph{Operators in subspaces.}

The theory of scalar operators, which we considered so far, has the following
natural generalization.

Let $D$ be a $N\times N$ matrix operator with matrix components, which are
noncommutative operators, and $P_1$ and $P_2$ be projections in the algebra
$\Mat_N(C^\infty(M)\rtimes\Ga)$  of  $N\times N$ matrices over the crossed
product $C^\infty(M)\rtimes\Ga$. Suppose that the following equality holds
$$
D=P_2DP_1.
$$
Then we can consider the restriction
\begin{equation}\label{eq-operator-in-sections}
    D:P_1(C^\infty(M,\mathbb{C}^N))\lra P_2(C^\infty(M,\mathbb{C}^N))
\end{equation}
of $D$ to subspaces in the space of vector-functions on  $M$, which are defined
as ranges of projections $P_{1}$ and $P_2$.

The operator~\eqref{eq-operator-in-sections} is denoted by
$\mathbf{D}=(D,P_1,P_2)$ and called \emph{noncommutative operator} acting in
spaces defined by projections $P_1,P_2$.
Operator~\eqref{eq-operator-in-sections} is called \emph{elliptic}, if there
exists a matrix symbol  $r\in\Mat_N(C^\infty(S^*M)\rtimes \Ga)$ such that
$$
r\sigma(D)=\sigma(P_1),\qquad \sigma(D)r=\sigma(P_2).
$$
In this case a result similar to Theorem~\ref{th-finite} holds. Namely, if
operator~\eqref{eq-operator-in-sections} is elliptic, then it has Fredholm
property and its kernel and cokernel consist of smooth functions.

\section{Topological invariants}

To describe index formula, we need to introduce topological invariants, which
describe contributions of the symbol of elliptic operator and the manifold.
There are two such invariants. The first is the Chern character, which is
defined on a certain $K$-group. The second invariant is a suitable modification
of the Todd class. These topological invariants are constructed in the present
section.

\subsection{$K$-theory class of elliptic symbol}

For a noncommutative elliptic operator $\mathbf{D}=(D,P_1,P_2)$, we define the
element
$$
[\sigma(\mathbf{D})]\in K_0(C^\infty_0(T^*M)\rtimes\Ga)
$$
of the even $K$-group of the crossed product of the algebra $C^\infty_0(T^*M)$
of compactly supported functions on the cotangent bundle by $\Ga$. To this end,
consider the projections
\begin{equation}\label{projector1} q_2=\frac12
\begin{pmatrix}
  (1-\sin\psi)\si(P_1)  &  \sigma(D)^{-1}\cos\psi \vspace{1mm}\\
  \sigma(D)\cos\psi  & (1+\sin\psi)\si(P_2)
\end{pmatrix},\quad
q_1= \begin{pmatrix}
  0 &  0 \\
  0 &  \si(P_2)
\end{pmatrix}
\end{equation}
over the algebra $C^\infty_0(T^*M)\rtimes\Ga$ with adjoint unit. Here $\psi\in
C^\infty(T^*M)$ denotes a $\Gamma$-invariant function on the cotangent bundle
$T^*M$ with canonical coordinates $x,\xi$, which depends only on $|\xi|$, for
small $|\xi|$ is equal to $-\pi/2$,  then increases as $|\xi|$ increases, while
for large $|\xi|$ this function is equal to $+\pi/2$. In addition, the symbols
of operators are considered as homogeneous functions of degree zero on the
cotangent bundle $T^*M\setminus 0$ minus the zero section. Finally, we assume
that the projection $P_2$ is equal to ${\rm
diag}(1,1,...,1,0,0,...,0)$.\footnote{This can be achieved, if we consider
direct sum of the initial operator and the invertible operator
$(Id-P_2,Id-P_2,Id-P_2)$.}

Let us now set
$$
 [\sigma(\mathbf{D})]=[q_2]-[q_1].
$$

\subsection{Chern character for crossed products}

Let $X$ be a compact smooth manifold, on which $\Ga$ acts and the action
satisfies conditions formulated at the beginning of Section~1.

To define the Chern character on the $K$-group of the crossed product
$C^\infty(X)\rtimes\Ga$, we embed the crossed product algebra in a certain
differential graded algebra (algebra of noncommutative differential forms), and
then construct a differential trace on this algebra, which takes noncommutative
differential forms to de Rham forms on the manifold.

\paragraph{Noncommutative differential forms.}

Denote by  $\Omega(X)$ the algebra of differential forms on $X$ with
differential $d$. The action of $\Gamma$ on smooth functions extends to the
action on forms.

Consider the smooth crossed product $\Omega(X)\rtimes \Ga$. This algebra is
graded and has the differential
$$
(d\om)(g)=d(\om(g)),\qquad \om\in \Omega(X)\rtimes \Ga.
$$
Equality $d^2=0$ is obvious, while the Leibniz rule
$$
d(\om_1\om_2)=(d\om_1)\om_2+(-1)^{{\rm deg}\om_1}\om_1 d\om_2,\quad
\om_1,\om_2\in \Omega(X)\rtimes \Ga,
$$
follows from the invariance of the exterior differential with respect to
diffeomorphisms.

We shall refer to $\Omega(X)\rtimes \Ga$ as the \emph{algebra of noncommutative
differential forms} on the $\Ga$-manifold $X$.

\paragraph{Differential trace.}

Given $g\in\Ga$, denote by $X^g$ the  fixed point set of the action of $g$ on
$X$. Since by assumption  $g$ acts isometrically, the set $X^g\subset
 X$ is a smooth submanifold (e.g., see~\cite{CoFl1}).

Let us define the \emph{differential trace}
\begin{equation}\label{eq-tau-g}
    \tau_g: \Om (X)\rtimes\Ga\lra \Om (X^g).
\end{equation}
This means that $\tau_g$ has to be linear mapping, which  satisfies the
following properties
\begin{equation}\label{eq-sled}
 \tau_g(\om_2\om_1)=(-1)^{\deg\om_1\deg\om_2}\tau_g(\om_1\om_2),\qquad  \text{for all }
 \om_1,\om_2\in\Om (X)\rtimes\Ga,
\end{equation}
\begin{equation}\label{eq-sled2}
 d\left(\tau_g(\om)\right)=\tau_g(d\om),\qquad  \text{for any form } \om\in\Om (X)\rtimes\Ga.
\end{equation}

To define the trace \eqref{eq-tau-g}, we introduce some notation. Let
$C_g\subset G$ be the centralizer\footnote{Note that the centralizer of an
element $g$ is the set of elements of the group, which commute with $g$.} of
$g$. The centralizer is a closed Lie subgroup in $G$. Denote elements of
centralizer by $h$ and the induced Haar measure on the centralizer by $dh$.

Denote by $\langle g\rangle\subset\Ga$  the conjugacy class of $g$, which is
the set of elements, which can be written in the form $zgz^{-1}$ for some $z\in
\Ga$. We also fix for any element $g'\in \langle g\rangle$ some element
$z=z(g,g')$, which conjugates  $g$ and $g'=zgz^{-1}$. Any such element defines
a diffeomorphism $z:X^g\to X^{g'}$.

Let us now define  the trace~\eqref{eq-tau-g} as
\begin{equation}\label{eq-sled-nash}
    \tau_g(\om)=
      \sum_{g'\in \langle g\rangle}\;\;\;
        \int_{C_g}
           \Bigl.h^*\bigl(
              {z}^*\om(g')
             \bigr)\Bigr|_{X^g}
           dh.
\end{equation}
One can show that this expression does not depend on the choice of elements $z$
and satisfies conditions \eqref{eq-sled},~\eqref{eq-sled2}.

\begin{example}
Suppose that $X$ is a one-point space, and $\Ga$ is finite. Then the crossed
product $\Omega(X)\rtimes\Ga$ coincides with the group algebra $\mathbb{C}\Ga$,
while the trace \eqref{eq-sled-nash} is equal to
$$
\tau_g(f)=\sum_{g'\in\langle g\rangle} f(g').
$$
\end{example}
\begin{example}
Suppose now that $X$ is arbitrary, and $\Ga$ is finite. Consider the unit
element $g=e\in\Ga$. In this case we can set $G=\Gamma$. Then the trace
$\tau_e$ is equal to
$$
\tau_e(\omega)=\frac 1{|\Gamma|}\sum_{h\in\Ga}  h^*(\omega(e)),\qquad \omega\in \Omega(X)\rtimes\Ga,
$$
i.e., its value is equal to the averaging over $\Ga$ of the value of $\omega$
on the unit element of the group.
\end{example}

\paragraph{Chern character.}

We are now in a position to define the Chern character. Let $P\in
\Mat_N(C^\infty(X)\rtimes\Ga)$ be a matrix projection.

Consider the restriction of the differential $d$ to the range of $P$:
$$
\nabla_P=P\circ d\circ P:\Omega(X,\mathbb{C}^N)\rtimes\Ga\lra
\Omega(X,\mathbb{C}^N)\rtimes\Ga.
$$
Here $P$ acts in the space $\Omega(X,\mathbb{C}^N)\rtimes\Ga$ by left
multiplication. An easy computation shows that the curvature form $ \nabla_P^2$
is actually operator of multiplication by a noncommutative matrix-valued
2-form.
\begin{definition}
The Chern character, corresponding to element $g\in\Ga$, of projection $P$ is
the following de Rham form
\begin{equation}\label{eq-chern}
    \ch_g(P)=\tr\tau_g\left(P\exp\left(-\frac{\nabla_P^2}{2\pi i}
    \right)\right)\in \Omega(X^g),
\end{equation}
defined on the fixed point set $X^g$, where $\tr$ is the matrix trace.
\end{definition}
A straightforward computation shows that the form~\eqref{eq-chern} is closed,
and its cohomology class does not change under stable homotopies of projection
$P$. Thus, we obtain the  well-defined homomorphism
\begin{equation}\label{eq-chern-parad}
    \begin{array}{ccc}
      \ch_g:K_0(C^\infty(X)\rtimes\Ga) & \lra & H^{ev}(X^g), \vspace{3mm}\\
      P & \longmapsto & \ch_g (P). \\
    \end{array}
\end{equation}

\begin{remark}
If the group is finite, then the Chern character~\eqref{eq-chern} is equal to
the one constructed in~\cite{Slo1},~\cite{BaCo2}.
\end{remark}

\paragraph{Chern character for projections in bundles.}

It is important in applications to have formulas for the Chern character of
projections defined in nontrivial vector bundles. Let us show how the formulas
given above have to be modified in this case.

Let $E$ be a finite-dimensional complex vector bundle over $X$. First, we
define the algebra of \emph{noncommutative endomorphisms} $C^\infty(X,\End
E)_\Ga$ as the algebra of functions on $\Ga$, such that the value of a function
at point $g\in\Ga$ is an element of the space $\Hom(E,g^*E)$. The product in
this algebra is defined by the formula
$$
(f_1f_2)(k)=\sum_{gh=k} h^*f_1(g)(h^*)^{-1}f_2(h),
$$
(cf. \eqref{eq-composition}) where $h^*:E\to h^*E$ is isomorphism. $\End
E$-valued noncommutative differential forms are defined similarly and denoted
by $\Omega(X,\End E)_\Ga$. The trace in this case is equal to
$$
\tau_g(\om)=
      \sum_{g'\in \langle g\rangle}\;\;\;
        \int_{C_g}
           \Bigl.h^*\tr\bigl(
              {z}^*\om(g')
             \bigr)\Bigr|_{X^g}
           dh,
$$
where $\tr$ is the trace of endomorphism of bundle $E|_{X^g}$.

Then the Chern character of any projection $P\in C^\infty(X,\End E)_\Ga$   is
computed by the same formula~\eqref{eq-chern}, where we use the following
noncommutative connection
$$
\nabla_P=P\circ \nabla_E\circ P:\Omega(X,\End E)_\Ga\lra \Omega(X,\End E)_\Ga.
$$
(Here $\nabla_E$ is  some connection    in $E$.)

\subsection{Todd class}

Given $g\in\Ga$,  denote by $M^g$ the fixed point manifold of this element. Let
$N^g$ be the normal bundle of $M^g$ in $M$.

The differential of $g$ induces orthogonal endomorphism of $N^g$ and the
exterior form bundle
$$
 \Omega(N^g\otimes\mathbb{C})=\Omega^{ev}(N^g\otimes\mathbb{C})\oplus
 \Omega^{odd}(N^g\otimes\mathbb{C}).
$$
Consider the Chern character\footnote{Recall (see~\cite{AtSi3}) the definition
of the class $\ch E(g)$ of a $G$-vector bundle $E$ on a trivial $G$-space $X$.
If we decompose the bundle $E$ as the direct sum $E=\oplus_\la E_\la$ of
eigensubbundles with regard to the action of $g$, then we have
$$
\ch E(g)=\sum_\la \la \ch E_\la\in H^{ev}(X).
$$
}
\begin{equation}\label{eq-forms}
\ch \Omega^{ev}(N^g\otimes\mathbb{C})(g)-\ch
\Omega^{odd}(N^g\otimes\mathbb{C})(g)\in H^{ev}(M^g).
\end{equation}
The zero degree component of this expression is nonzero~\cite{AtSe2}.
Therefore, the class~\eqref{eq-forms} is invertible and the following class
\begin{equation}\label{eq-todd-class}
    \Td_g(T^*M\otimes \mathbb{C})=\frac{\Td(T^*M^g\otimes\mathbb{C})}{\ch \Omega^{ev}(N^g\otimes\mathbb{C})(g)-\ch
\Omega^{odd}(N^g\otimes\mathbb{C})(g)},
\end{equation}
is well defined, where $\Td$ in the right hand side of the equality is the
standard Todd class of a complex vector bundle, and the division is well
defined, because both  numerator and denomenator are even degree cohomology
classes.

\section{Index theorem}

\begin{theorem}\label{atsi}
Let $D$ be a noncommutative elliptic operator on a closed manifold $M$. Then
\begin{equation}\label{Cohom1}
\ind D=\sum_{\langle g \rangle\subset\Ga}\left\langle
\ch_g[\sigma(D)]\Td_g(T^*M\otimes\mathbb{C}), [T^*M^{g}]\right\rangle,
\end{equation}
where $\langle g\rangle$ runs over the set of conjugacy classes of   $\Ga$\rom;
$[T^*M^{g}]\in H_{ev}(T^*M^{g})$ is the fundamental class of $T^*M^{g}$, the
Todd class is lifted from  $M_{g}$ to the corresponding $T^*M^{g}$ by the
natural projection and the angular brackets denote the pairing between homology
and cohomology. The series in~\eqref{Cohom1} is absolutely convergent.
\end{theorem}

For matrix operators (systems of equations), the index formula is written in
terms of integrals over the corresponding cosphere bundles.  To formulate the
result, we introduce the corresponding odd Chern character.

\begin{definition}\label{cosphere}
The \emph{odd Chern character} of matrix elliptic symbol
$$
\sigma\in
\Mat_N(C^\infty(S^*M)\rtimes \Gamma)
$$
is the collection
$$
\ch^{odd}(\sigma)\in \bigoplus_{\langle g\rangle\subset\Ga} H^{odd}(S^*M^g)
$$
of cohomology classes, which are defined as
\begin{equation}\label{eq-chern-odd}
\ch^{odd}_g(\sigma)= \tr\tau_g\BL[
 \sum_{n\ge 0}\frac{n!}{(2\pi i)^{n+1}(2n+1)!}
(\sigma^{-1}d\sigma)^{2n+1} \BR].
\end{equation}
\end{definition}
The index formula~\eqref{Cohom1} in this case becomes
\begin{equation}\label{eq-index-1a}
\ind D=\sum_{\langle g \rangle\subset\Ga}\left\langle \ch^{odd}_g (\sigma(D))
\Td_g(T^*M\otimes\mathbb{C}), [S^*M_{g}]\right\rangle.
\end{equation}

\begin{theorem}[on index contribution of trivial element of the group]\label{th-e}
Suppose that either the action of $\Ga$ on $M$ is free or $\Ga$ is torsion
free. Then for any elliptic operator $D$ one has
\begin{equation}\label{eq-index-2}
    \ind D=\langle \ch_e[\sigma(D)]\Td(T^*M\otimes\mathbb{C}),[T^*M]\rangle.
\end{equation}
\end{theorem}
\begin{proof}
If the action is free, Eq.\eqref{eq-index-2} follows from~\eqref{Cohom1}, since
 $M^g$ is empty for any $g\ne e$.

Suppose now that  $\Ga$ is torsion free (i.e., has no elements of finite
order). Denote by  $C^\infty(\Ga)$ the smooth crossed product
$\mathbb{C}\rtimes\Ga$, which is just the algebra of rapidly decaying functions
on $\Ga$ with convolution product.

1.  It was shown in~\cite{NaSaSt17}, that the index mapping can be included as
a side of the commutative triangle
\begin{equation}\label{eq-triangle1}
    \xymatrix{
     K_0(C_0^\infty(T^*M)\rtimes\Ga)\ar[rr]^{p_!}\ar[dr]_{\ind} & &
     K_0(C^\infty(\Ga))\ar[dl]^\tau,\\
     & \mathbb{Z}&
    }
\end{equation}
where $p_!$ stands for the direct image mapping in $K$-theory, which is induced
by the projection $p:M\to pt$ to the one-point space, while
\begin{equation}\label{eq-sled1}
    \tau=\sum_{\langle g\rangle\subset\Ga}\ch_g,
\end{equation}
where we recall that $\ch_g:K_0(C^\infty(\Gamma))\to H^{ev}(pt)=\mathbb{C}$. In
addition, the number $\ch_g(p_![\sigma(D)])$ is equal to the contribution of
the conjugacy class $\langle g\rangle$ to the index formula~\eqref{Cohom1}.

2. The mapping $\tau$ is induced by the tracial state
$$
 \begin{array}{cccc}
   \tau:&C^*(\Ga)  & \lra & \mathbb{C}, \vspace{2mm}\\
    &f  & \longmapsto & \sum\limits_{g\in\Gamma} f(g), \\
 \end{array}
$$
where $C^*(\Ga)$ stands for the group $C^*$-algebra. Consider now the tracial
state
$$
\tau_e=\ch_e,
$$
i.e., $\tau_e(f)=f(e)$. By~\cite[Theorem 1]{Ji1} all tracial states induce the
same mapping on the $K$-group. The desired equality
$$
\ind D=\tau(p_![\sigma(D)])=\tau_e(p_![\sigma(D)])
$$
now follows, i.e., the index is equal to the contribution of the trivial
conjugacy class.

The proof of the theorem  is now complete.
\end{proof}

\section{Index of twisted operators}

In this section we show how $G$-invariant elliptic operators can be twisted by
certain projections to produce noncommutative elliptic operators.

\paragraph{Twisted operators.}

Suppose that we have a $G$-invariant elliptic operator
\begin{equation*}\label{eq-oper-G-inv}
    D:C^\infty(M,E)\lra C^\infty(M,F)
\end{equation*}
acting in spaces of sections of two $G$-bundles $E,F$. The $G$-invariance
condition means that for any $g\in G$ one has
$$
DT_E(g)=T_F(g)D,
$$
where $T_{E}(g)$, $T_F(g)$ denote actions of shift operators on sections of $E$
and $F$, respectively.

Given a projection
\begin{equation*}\label{eq-projection1}
    P\in\Mat_n(C^\infty(M)\rtimes\Ga),
\end{equation*}
let us define projection $\widetilde{P}_E:C^\infty(M,E\otimes \mathbb{C}^n)\lra
C^\infty(M,E\otimes \mathbb{C}^n)$ by the formula
\begin{equation}\label{eq11}
    P_E=\sum_{g\in\Gamma} T_E(g)\otimes P(g).
\end{equation}
One defines a similar projection  $P_F$, corresponding to $F$.

The direct sum of $n$ copies of $D$ is denoted by $D\otimes 1_n$.
\begin{definition}
The operator
\begin{equation}\label{eq-oper-coeff}
   {P}_F(D\otimes 1_n){P}_E:{P}_E(C^\infty(M,E\otimes
    \mathbb{C}^n))\lra {P}_F(C^\infty(M,F\otimes \mathbb{C}^n))
\end{equation}
is denoted by  $D\otimes 1_P$ and called \emph{operator $D$ twisted by
projection $P$.}
\end{definition}

\begin{proposition}
The operator~\eqref{eq-oper-coeff} is elliptic.
\end{proposition}
\begin{proof}
Consider the symbol
\begin{equation}\label{eq-inverse}
\sigma({P}_E)(\sigma(D)^{-1}\otimes 1_n)\sigma({P}_F).
\end{equation}
Here the inverse $\sigma(D)^{-1}$ exists, since we assumed that  $D$ is
elliptic.

We claim that the symbol~\eqref{eq-inverse} is the inverse of the symbol of the
operator~\eqref{eq-oper-coeff}. Indeed, it follows from the $G$-invariance of
$D$ that $\sigma(D\otimes 1_n)$ intertwines projections $\sigma({P}_E)$ and
$\sigma({P}_F)$. Thus, we get
$$
\sigma({P}_F)(\sigma(D)\otimes 1_n)\sigma({P}_E)(\sigma(D)^{-1}\otimes
1_n)\sigma({P}_F)= \sigma({P}_F)(\sigma(D)\otimes 1_n)(\sigma(D)^{-1}\otimes
1_n)\sigma({P}_F)= \sigma({P}_F).
$$
Similar computation can be done for the composition in inverse order.
\end{proof}

\paragraph{Multiplicative property of Chern character.}

The symbol of the  $G$-invariant elliptic operator $D$ defines the class
$$
 [\sigma(D)]\in K^0_{G,c}(T^*M)
$$
in the equivariant $K$-group  (with compact supports) of the cotangent bundle,
while projection $P$ defines the class
$$
 [P]\in K_0(C^\infty(M)\rtimes\Ga)
$$
in the $K$-group of the crossed product.

The mapping $\sigma(D),P\mapsto \sigma(D\otimes 1_P)$, where $D\otimes 1_P$ is
the twisted operator, defines the product on $K$-groups
\begin{equation*}\label{eq-k-product}
    \begin{array}{ccc}
      K_0(C^\infty(M)\rtimes\Ga)\times K^0_{G,c}(T^*M) & \lra & K_0(C^\infty_0(T^*M)\rtimes\Ga), \vspace{2mm} \\
      \left[P\right],  [\sigma(D)] & \longmapsto & [\sigma(D\otimes 1_P)]. \\
    \end{array}
\end{equation*}

To formulate the multiplicative property of the Chern character with respect to
this product, we denote by   $G(g)\subset G$ the compact subgroup generated by
element $g$, and by $ev_g:R(G(g))\to\mathbb{C}$ the evaluation mapping for
characters of representations at point $g$. Let us define the homomorphism
$$
 \ch(g):K^0_G(X)\to H^{ev}(X^g)
$$
as the composition
\begin{equation}\label{localiz}
K^0_G(X)\to K^0_{G(g)}(X^g)\simeq K^0(X^g)\otimes R(G(g)) \stackrel{ev_g}\to
K^0(X^g)\stackrel{\ch}\to  H^{ev}(X^g),
\end{equation}
where the first mapping is the restriction to the fixed point set and the
isomorphism describes the equivariant $K$-group of a trivial $G(g)$-space in
terms of usual (nonequivariant) $K$-groups.

It was proved in \cite{NaSaSt17} that the diagram
\begin{equation}
\label{prodacc22}
 \xymatrix{
   K_0(C^\infty(M)\rtimes\Ga)\times K^0_{G,c}(T^*M) \ar[r] \ar[d]_{\ch_g\times \ch(g)} &
   K_0(C^\infty_0(T^*M)\rtimes\Ga)\ar[d]^{\ch_g}\\
   H^{ev}(M^g)\times H^{ev}_c(T^*M^g) \ar[r] & H^{ev}_c(T^*M^g)
 }
\end{equation}
is commutative for any $g\in\Ga$.

\paragraph{Index formula for twisted operators.}

We shall assume for simplicity that all submanifolds of fixed points   $M^g$
are orientable.\footnote{If the fixed-point manifolds are nonorientable, then
the formulas below can be modified  using local coefficient systems as in
\cite{AtSi3}. We shall omit this standard procedure here.}

Since the diagram~\eqref{prodacc22} is commutative, we have
$$
\ch_g[\sigma(D\otimes 1_P)]=\ch[\sigma(D)](g)\cdot \ch_g(P).
$$
Substituting this equality in the index formula~\eqref{Cohom1}, and integrating
over the fibers of the cosphere bundles, we obtain the   formula
\begin{equation}\label{Cohom1a}
\ind (D\otimes 1_P)=\sum_{\langle g \rangle\subset\Ga}\bigl\langle
\psi^{-1}_g\left(\ch[\sigma(D)](g)\right)\Td_g(T^*M\otimes\mathbb{C})\ch_g(P),
[M^{g}]\bigr\rangle
\end{equation}
for the index of twisted operators, where $\psi_g: H^*(M^g)\lra H^*_c(T^*M^g)$
stands for the Thom isomorphism defined by the orientation of $M^g$.

\section{Index of geometric operators}

Let us give index formulas for geometric operators  twisted by noncommutative
projections.

\paragraph{Euler operator.}

Let $\mathcal{E}$ be the Euler operator
\begin{equation*}\label{eq-euler}
    d+d^*:\Omega^{ev}(M)\lra \Omega^{odd}(M),
\end{equation*}
acting between spaces of forms of even and odd degrees.

Consider the twisted Euler operator\footnote{Hereinafter, lifts
(see~\eqref{eq11}) of projection $P$ to spaces of vector bundle sections will
be denoted for short as $P$.}
\begin{equation*}\label{eq-euler-twist}
    \mathcal{E}\otimes 1_P:  {P}\Omega^{ev}(M,\mathbb{C}^n)\lra
     {P}\Omega^{odd}(M,\mathbb{C}^n),
\end{equation*}
where $P\in\Mat_n(C^\infty(M)\rtimes\Ga)$ is a projection.

\begin{theorem}\label{th-euler-index} The index of twisted Euler operator is
equal to
$$
 \ind (\mathcal{E}\otimes 1_P)=\sum_{\langle g\rangle\subset\Ga}\chi(M^g)\ch_g^0(P),
$$
where  $\chi$ stands for the Euler characteristic, the number $\ch_g^0(P)$ is
equal to the zero degree component of the Chern character. In addition,
$\chi(M^g)$ and $\ch_g^0(P)$ are treated as locally-constant functions on the
fixed point set $M^g$.
\end{theorem}
\begin{proof}
1. The restriction of the symbol $\sigma(\mathcal{E})$ to $T^*M^g\subset T^*M$
has the decomposition%
\footnote{This decomposition follows from the orthogonal decomposition of
tangent bundle
$$
 TM|_{M^g}\simeq TM^g\oplus N,
$$
and the corresponding graded decompositions $\Omega(TM)|_{M^g}\simeq
\Omega(TM^g)\otimes \Omega(N)$ of the exterior forms.}
\begin{equation}\label{eq-euler-symbol}
    \sigma(\mathcal{E})|_{T^*M^g}\simeq \Bigl(\sigma(\mathcal{E}_{M^g})\otimes 1_{\Omega^{ev}(N)}\Bigr)\oplus
    \Bigl(\sigma(\mathcal{E}^*_{M^g})\otimes 1_{\Omega^{odd}(N)}\Bigr),
\end{equation}
where $\mathcal{E}_{M^g}$ denotes the Euler operator on $M^g$ and
$\mathcal{E}^*_{M^g}$ is its adjoint.

2. It follows from~\eqref{eq-euler-symbol} that
$$
\ch\sigma(\mathcal{E})(g)=\ch\sigma(\mathcal{E}_{M^g})\cdot
\ch(\Omega^{ev}(N)-\Omega^{odd}(N))(g).
$$
Hence, we obtain
$$
 \ch\sigma(\mathcal{E})(g)\Td_g(T^*M\otimes
 \mathbb{C})=\ch\sigma(\mathcal{E}_{M^g})\Td (T^*M^g\otimes
 \mathbb{C}).
$$
Substituting this formula in Eq.~\eqref{Cohom1a}, we see that the contribution
of $M^g$ to the index is equal to
$$
\int_{M^g}\psi^{-1}_g\left(\ch\sigma(\mathcal{E}_{M^g})(g)\right) \Td
(T^*M^g\otimes  \mathbb{C})\ch_g (P)=\int_{M^g}
e(TM^g)\ch_g(P)=\chi(M^g)\ch_g^0(P),
$$
where $e(TM^g)$ stands for the Euler class. Here in the first equality we use
the classical result
$$
 \psi^{-1}_g\left(\ch\sigma(\mathcal{E}_{M^g})(g) \right)\Td (T^*M^g\otimes
 \mathbb{C})=e(TM^g),
$$
and in the second equality the fact that the Euler class is a top degree class.

The proof of the theorem is now complete.
\end{proof}

\paragraph{Signature operator.}

Suppose that $M$ is a $2l$-dimensional oriented manifold. We define the
involution on the space of differential forms
$$
\alpha=i^{k(k-1)+l}*:\Omega^k(M)\lra  \Omega^{2l-k}(M),\qquad \alpha^2=Id,
$$
where $*$ stands for the Hodge operator. The subspace of forms satisfying
equality   $\alpha\om=\om$ is called the space of \emph{self-dual forms}, and
satisfying equality $\alpha\om=-\om$ is called \emph{anti-self-dual forms}. The
corresponding subspaces in  $\Omega(M)$ are denoted by $\Omega^+(M)$ and
$\Omega^-(M)$.

Denote by $\mathcal{S}$ the signature operator
\begin{equation}\label{eq-signature}
    d+d^*:\Omega^{+}(M)\lra \Omega^{-}(M),
\end{equation}
acting from self-dual to anti-self-dual forms. Let us  assume that $G$ acts on
$M$ by orientation-preserving diffeomorphisms. In this case $\Omega^{\pm}(M)$
are $G$-invariant subspaces. Consider the twisted signature operator
\begin{equation}\label{eq-sign-twist}
    \mathcal{S}\otimes 1_P: {P}\Omega^+(M,\mathbb{C}^n)\lra
    {P}\Omega^-(M,\mathbb{C}^n),
\end{equation}
where $P\in\Mat_n(C^\infty(M)\rtimes\Ga)$ is a projection.

To give explicit index formula in this case, let us introduce
following~\cite{AtSi3} special cohomology class. Choose an element $g\in\Ga$.
Then $g$ acts on the normal bundle  $N^g$ of the embedding $M^g\subset M$ by
orthogonal transformations. Therefore, $N^g$ decomposes as the orthogonal sum
$$
N^g=N^g(-1)\bigoplus_{0<\theta<\pi} N^g(\theta),\qquad \text{where }
g|_{N^g(-1)}=-Id,\;\; g|_{N^g(\theta)}=e^{i\theta}Id,
$$
of eigensubbundles. The bundles $N^g(\theta)$ have complex structure. Let us
note that there is no eigenvalue $1$ in this decomposition. The bundle
$N^g(-1)$, as well as the manifold $M^g$ are even-dimensional. Denote
$$
t=\dim M^g/2,\quad r=\dim N^g(-1)/2,\quad s(\theta)=\dim N^g(\theta)/2.
$$

Let $L_g(M)$ be the cohomology class
\begin{equation*}\label{eq-class3}
2^{t-r}\prod_{0<\theta<\pi}
                     \left(\bigl(i{\rm tg}\frac\theta 2\bigr)^{-s(\theta)}\right)
                     L(TM^g)
                     L(N^g(-1))^{-1}
                     e(N^g(-1))
                     \prod_{0<\theta<\pi} \mathcal{M}^\theta(N^g(\theta))\in H^{ev}(M^g),
\end{equation*}
where $L$ is the  Hirzebruch $L$-class of a real vector bundle, defined in the
Borel-Hirzebruch formalism by the function
$$
\frac{\displaystyle   x/ 2}{{\rm th }\displaystyle\;  x /2},
$$
$e(N^g(-1))$ is the Euler class of an oriented vector bundle,\footnote{The
orientation of the bundle $N^g(-1)$ is determined by the orientation of $M^g$
(see~\cite{AtSi3}).} and $\mathcal{M}^\theta(N^g(\theta))$ is the
characteristic class of the complex vector bundle $N^g(\theta)$, which is
defined by the function
$$
\frac{{\rm th}\displaystyle\frac {i\theta}2}{{\rm th}\displaystyle\frac
{x+i\theta}2}.
$$

\begin{theorem}\label{th-signature-index} The index of twisted signature
operator is equal to
\begin{equation*}
 \ind (\mathcal{S}\otimes 1_P)=
  \sum_{\langle g\rangle\subset\Ga}\;\;
   \int_{M^g} L_g(M)
                     \ch_g(P).
\end{equation*}

\end{theorem}
\begin{proof}
Atiyah and Singer in~\cite{AtSi3} proved that
\begin{equation*}
\psi^{-1}_g\left(\ch\sigma(\mathcal{S})(g)\right)\Td_g(T^*M\otimes
\mathbb{C})=L_g(M).
\end{equation*}
Substituting this equality in Eq.~\eqref{Cohom1a}, we see that the contribution
of   $M^g$ to the index is equal to the desired expression.
\end{proof}

\paragraph{Dirac operator.}

Let $M$ be an even-dimensional oriented manifold, which is  endowed with a
$G$-invariant spin-structure (i.e., the action of $G$ on $M$ lifts to the
action on the spinor bundle $S(M)$). Let $\mathcal{D}$ be the Dirac
operator~\cite{AtSi3}
\begin{equation*}\label{eq-dirac}
    \mathcal{D}:S^+(M)\lra S^-(M),
\end{equation*}
acting between sections of half-spinor bundles.

Consider the twisted Dirac operator
\begin{equation*}\label{eq-dirac-twist}
    \mathcal{D}\otimes 1_P:  {P}S^+(M,\mathbb{C}^n)\lra
     {P}S^-(M,\mathbb{C}^n),
\end{equation*}
where $P\in\Mat_n(C^\infty(M)\rtimes\Ga)$ is some projection.

Let us consider for simplicity the case when either $\Ga$ is torsion free, or
its action is free. In this case, Eqs.~\eqref{eq-index-2} and \eqref{Cohom1a}
give the index formula.
\begin{theorem}\label{th-dirac-index} The index of twisted Dirac operator is
equal to
\begin{equation}\label{eq-dirac-index}
 \ind (\mathcal{D}\otimes 1_P)=\int_M A(TM)\ch_e(P),
\end{equation}
where $A(TM)$ is the $A$-class of the tangent bundle, which is defined in
Borel-Hirzebruch formalism by the function
$$
\frac{x/2}{{\rm sh}\;x/2}.
$$
\end{theorem}

\section{Index of operators on the circle}

\paragraph{Reflection group.}

Consider the operator
\begin{equation}\label{eq-oper-circle-1}
    D=\frac d{d\varphi}:C^\infty_+(\mathbb{S}^1)\lra
    C^\infty_-(\mathbb{S}^1),
\end{equation}
acting in the spaces
$$
C^\infty_{\pm}(\mathbb{S}^1)=\{u\in C^\infty(\mathbb{S}^1) \;|\;
u(-\varphi)=\pm u(\varphi)\}
$$
of even and odd functions on the circle. On the one hand, the index of this
operator is obviously equal to one.

On the other hand, the operator \eqref{eq-oper-circle-1} is equivalent to the
Euler operator
\begin{equation*}\label{eq-oper-circle-1a}
  \begin{array}{cccc}
    \mathcal{E}\otimes 1_P:& P\Omega^0(\mathbb{S}^1) & \lra  & P\Omega^1(\mathbb{S}^1), \vspace{2mm}\\
    &f & \longmapsto & df, \\
  \end{array}
\end{equation*}
twisted by the projection $P=(1+g^*)/2$, where $g(\varphi)=-\varphi$ denotes
the action of $\mathbb{Z}_2$ by reflections.

The mapping $g$ has 2 fixed points: $\varphi=0$ and $\varphi=\pi$. Thus, by
virtue of Theorem~\ref{th-euler-index} we get the index formula
\begin{multline*}
\ind D=\ind (\mathcal{E}\otimes 1_P)
=\chi(\mathbb{S}^1)\ch_e^0(P)+\chi(\{0\})\ch^0_g(P) +\chi(\{\pi\})\ch^0_g(P)
=\\
=
0+\frac 1 2+\frac 1 2=1,
\end{multline*}
where $\ch^0_g$ stands for the zero component of the Chern character.

This example shows that the contributions to the index formula of the conjugacy
classes of the group  can be fractional numbers.

\paragraph{Rotation group.}

Given a nonzero vector
$$
 \alpha=(\alpha_1,...,\alpha_n)\in \mathbb{T}^n,
$$
let us consider the equation
\begin{equation}\label{eq-circle-2}
\sum_{g\in \mathbb{Z}^n,l\le m}a_{g,l}(x)\left(-i\frac
d{dx}\right)^lu(x-g\alpha)=f(x)
\end{equation}
on the circle $\mathbb{S}^1$, where $u,f\in C^\infty(\mathbb{S}^1)$,  and
$g\alpha=g_1\alpha_1+\ldots+g_n\alpha_n.$ Here we assume that the coefficients
$a_{g,l}(x)$ are rapidly decaying as $|g|\to\infty$.

The vector $\alpha$ defines the action of the group $\mathbb{Z}^n$ on
$\mathbb{S}^1$:
$$
 g(\varphi)=\varphi+g\alpha.
$$
Denote the corresponding smooth crossed product by
$C^\infty(\mathbb{S}^1)\rtimes \mathbb{Z}^n$. Eq.~\eqref{eq-circle-2}
corresponds to the following noncommutative operator\footnote{The coefficients
in~\eqref{eq-circle-2} and~\eqref{eq-nc-operator2} are related by the equality:
$b_{g,l}(x)=a_{g,l}(x+g\alpha)$.}
\begin{equation}\label{eq-nc-operator2}
    D=\sum_{g\in \mathbb{Z}^n,l\le m}T(g)b_{g,l}(x)\left(-i\frac
d{dx}\right)^l:C^\infty(\mathbb{S}^1)\lra  C^\infty(\mathbb{S}^1).
\end{equation}
The symbol of $D$ is equal to
\begin{equation}\label{eq-nc-symbol1}
    \sigma(D)=\{b_{g,m}(x)\xi^m\}_{g\in \mathbb{Z}^n}\in C^\infty(S^*\mathbb{S}^1)\rtimes
    \mathbb{Z}^n, \quad (x,\xi)\in S^*\mathbb{S}^1.
\end{equation}
The decomposition $S^*\mathbb{S}^1=\mathbb{S}^1 \sqcup \mathbb{S}^1$ of the
cosphere bundle induces the decomposition
$$
C^\infty(S^*\mathbb{S}^1)\rtimes
    \mathbb{Z}^n    =C^\infty(\mathbb{S}^1)\rtimes \mathbb{Z}^n
    \oplus C^\infty(\mathbb{S}^1)\rtimes \mathbb{Z}^n.
$$
The corresponding components of the symbol are denoted by
$\sigma(D)=\sigma_+\oplus\sigma_-.$

If the symbol is invertible in  $C^\infty(S^*\mathbb{S}^1)\rtimes
\mathbb{Z}^n$, then the operator is elliptic. Since the group $\mathbb{Z}^n$ is
torsion free, there are no index contributions   of nontrivial conjugacy
classes and we get the  index formula
\begin{equation*}\label{eq-index-on-circle}
    \ind D=-\frac 1{2\pi i}\int_{\mathbb{S}^1}
      \Bigl[
        (\sigma_+^{-1}d\sigma_+)(0)-
        (\sigma_-^{-1}d\sigma_-)(0)
      \Bigr]
\end{equation*}
in terms of the symbol. Here we applied the index formula for matrix
operators~\eqref{eq-index-1a}.

The obtained index formula is similar to N\"other-Muskhelishvili formula, which
expresses the index of singular integral operators on the circle in terms of
the winding number of the symbol.

\begin{remark}
If the components of  $\alpha$ in \eqref{eq-circle-2} linearly dependent over
the field of rational numbers, then there are relations among the corresponding
shifts. In this case, it is natural to formulate the ellipticity condition in
terns of the crossed product
$C^\infty(S^*\mathbb{S}^1)\rtimes(\mathbb{Z}^n/L),$ where $L$ is a certain
lattice and there may be nontrivial contributions of fixed point sets to the
index formula.
\end{remark}

\section{Index of operators on the real line}

A.~Connes in his note~\cite{Con4} considered   elliptic operators on the
noncommutative torus. Later, such operators were studied in a number of papers
(e.g., see \cite{Baaj1,Baaj2,Con1,Ros6}). Let us show how the results presented
in the preceding sections enable one to compute index in this situation.

\paragraph{Noncommutative torus $A^\infty_{1/\theta}$.}

Given some number $\theta$ such that $0<\theta\le 1$, consider operators
\begin{equation}\label{eq-nctor-1}
    (Uf)(x)=f(x+1),\qquad (Vf)(x)=e^{-2\pi ix/\theta}f(x)
\end{equation}
of unit shift and multiplication by exponential, which act on functions on the
real line with coordinate $x$. These operators satisfy the commutation relation
\begin{equation}\label{nctorus-comm}
    VU=e^{2\pi i/\theta}UV.
\end{equation}
Let $A_{1/\theta}^\infty$ denote the algebra of operators
\begin{equation}\label{sum6}
\sum_{j,k=-\infty}^{\infty}a_{jk} U^jV^k,
\end{equation}
where coefficients $a_{jk}\in \mathbb{C}$ are rapidly decaying: for any $l$
there exists a constant $C$ such that
$$
|a_{jk}|\le C(1+|j|+|k|)^{-l},\qquad j,k=0,\pm1,\pm2,\dots
$$

\begin{example}
For $\theta=1$  the algebra $A_{1/\theta}^\infty$ is commutative and isomorphic
to the algebra of smooth functions on the torus $\mathbb{T}^2$. The isomorphism
is defined on generators as
$$
U\mapsto e^{i\varphi},\qquad V\mapsto e^{i\psi},
$$
where $\varphi$, $\psi$ denote coordinates on $\mathbb{T}^2$. In addition, the
rapid decay condition on $a_{jk}$ is transformed to the condition of
$C^\infty$-smoothness of the corresponding function on the torus.
\end{example}

For $\theta\ne 1$  the algebra $A_{1/\theta}^\infty$ is noncommutative and is
called the  \emph{algebra of functions on the noncommutative torus}.

\begin{remark}
The fact that $A^\infty_{1/\theta}$ is indeed an algebra, i.e., the product of
two series~\eqref{sum6} enjoys the rapid decay property, follows from the
isomorphism of $A^\infty_{1/\theta}$ and the smooth crossed product
$C^\infty(\mathbb{S}^1)\rtimes \mathbb{Z}$, where $\mathbb{Z}$ acts by
rotations by angle $2\pi/\theta$ (so-called irrational rotation algebra).
\end{remark}

\paragraph{Operators on the noncommutative torus.
Reduction to operators on $\mathbb{T}^2$.}

In the Schwartz space $S(\mathbb{R})$  of functions on the real line, A.~Connes
proposed to consider differential operators\footnote{The reader might naturally
have a question, why operators of the form~\eqref{eq-operators-connes} are
called operators  \emph{on} noncommutative torus? More generally, what are
operators on a ``noncommutative space''? We refer the reader to~\cite{Con1} for
detailed exposition of these questions. Here we note only, that
operators~\eqref{eq-operators-connes} have compact commutators with operators
$U',V'$, which generate the noncommutative torus $A^\infty_{ \theta}$:
\begin{equation*}
    (U'f)(x)=f(x+\theta),\qquad (V'f)(x)=e^{-2\pi ix}f(x).
\end{equation*}
In addition, if the operator has the Fredholm property, then it defines an
element in the Kasparov $K$-homology group of the algebra $A^\infty_{ \theta}$.
It is shown in the cited book that the mentioned duality between the algebras
 $A^\infty_{1/\theta}$ and  $A^\infty_{
\theta}$ is actually a special case of Poincar\'e duality in noncommutative
geometry.}
\begin{equation}\label{eq-operators-connes}
    D=\sum_{\alpha+\beta\le m} a_{\alpha\beta}x^\alpha\left(-i\frac
    d{dx}\right)^\beta:S(\mathbb{R})\lra S(\mathbb{R}),
\end{equation}
whose coefficients contain polynomials and elements $a_{\alpha\beta}\in
A^\infty_{1/\theta}$.

Let us show how operators~\eqref{eq-operators-connes} can be reduced to
operators on a compact manifold.

Consider the real line as the total space of the spiral covering
$$
\mathbb{R}\lra \mathbb{S}^1
$$
over the circle of length $\theta$ (see Fig.~\ref{fi2}).
\begin{figure}
  \begin{center}
   \includegraphics[height=6cm]{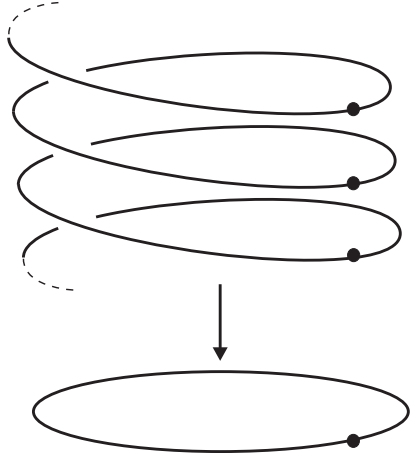}
   \caption{The line covers the circle}\label{fi2}
  \end{center}
\end{figure}

Then the Schwartz space on $\mathbb{R}$ is isomorphic to the space of smooth
sections of a (nontrivial!) bundle over  $\mathbb{S}^1$, whose fiber is the
Schwartz space $S(\mathbb{Z})$ of rapidly-decaying sequences (i.e., functions
on fibers  of the covering). If we now make Fourier transform $\mathcal{F}$
$$
S(\mathbb{Z})\stackrel{\mathcal{F}}\lra C^\infty(\mathbb{S}^1)
$$
fiberwise, then the obtained space is just the space of smooth sections of a
one-dimensional complex vector bundle over the torus   $\mathbb{T}^2$.
Collecting these transformations, we obtain the following result.

On $\mathbb{T}^2$ choose coordinates $0\le\varphi\le\theta$, $0\le\psi\le 1$.
\begin{lemma}
One has isomorphism
\begin{equation}\label{eq-isom1}
    S(\mathbb{R})\simeq C^\infty(\mathbb{T}^2,\gamma)
\end{equation}
of the Schwartz space on the real line and the space
\begin{equation*}\label{bottbundle}
C^\infty(\mathbb{T}^2,\gamma)=\{ g\in C^\infty(\mathbb{R}\times \mathbb{S}^1)
\; |\; g(\varphi+\theta,\psi)=g(\varphi,\psi)e^{-2\pi i\psi}\}
\end{equation*}
of smooth sections of the line bundle $\ga$ on $\mathbb{T}^2$. This isomorphism
is defined as
$$
f(x)\longmapsto \sum_{n\in\mathbb{Z}} f(\varphi+\theta n)e^{2\pi i n\psi}.
$$
\end{lemma}

The inverse mapping is
$$
g(\varphi,\psi)\longmapsto \frac 1
{2\pi}\int_{\mathbb{S}^1}g(\theta\{x/\theta\},\psi)e^{-2\pi
i[x/\theta]\psi}d\psi,
$$
where $[a]$ and $\{a\}$ are, respectively,  the integer and fractional part of
a real number $a$.

Applying isomorphism \eqref{eq-isom1}, we obtain the following correspondence
between   operators on the real line and on the torus
\begin{equation*}\label{tab-correspondence}
    \begin{tabular}{|c|c|}
      \hline
      operators on the real line & operators on the torus \\
      \hline
       $ -i\frac d{dx}$ & $-i\frac \partial{\partial\varphi}$\phantom{$\displaystyle\frac ZZ$}\\
      \hline $x$ &  $-i\frac\theta{2\pi}\frac \partial{\partial\psi}+\psi$ \phantom{$\displaystyle\frac ZZ$}\\
      \hline  $ e^{-2\pi ix/\theta} $ &  $  e^{-2\pi i\varphi/\theta} $\phantom{$\displaystyle\frac ZZ$} \\
      \hline $f(x)\to f(x+1)$& $g(\varphi,\psi)\mapsto g(\varphi+1,\psi)$\phantom{$\displaystyle\frac ZZ$}\\
      \hline
    \end{tabular}
\end{equation*}
It follows from this table that we obtain noncommutative operators on the
torus, which can be studied using Theorems~\ref{th-finite} and \ref{atsi}. This
was done in \cite{NaSaSt17} (including the computation of the index in terms of
the symbol of operator). In the present paper, we restrict ourselves by
considering one important example.

\paragraph{Index of twisted Dirac operators.}
Let $P\in A^\infty_{1/\theta}$ be a projection. Consider the noncommutative
operator
\begin{equation}\label{eq-oper-connes}
    D=P\left(x+\frac d{dx}\right)P: P(S(\mathbb{R}))\lra  P(S(\mathbb{R})).
\end{equation}
The corresponding operator on the torus is equal to
\begin{equation}\label{op-na-tore}
P\left(-i\frac\theta{2\pi}\frac \partial{\partial\psi}+\frac
\pa{\pa\varphi}+\psi\right)P: P(C^\infty(\mathbb{T}^2,\gamma))\lra
P(C^\infty(\mathbb{T}^2,\gamma))
\end{equation}
and is the adjoint of the Dirac operator
$$
-i\frac\theta{2\pi}\frac \partial{\partial\psi}-\frac \pa{\pa\varphi},
$$
twisted by projection $P$. Therefore, operator \eqref{eq-oper-connes} is
Fredholm for any projection $P$.

Let us note that the operator \eqref{op-na-tore} acts in sections of a
nontrivial bundle $\ga$. Thus, to compute the index, we shall apply index
formula \eqref{eq-dirac-index}.

To this end, we choose the following connection in $\gamma$
$$
\nabla_\ga=d+\frac{2\pi i}\theta\varphi d\psi:C^\infty(\mathbb{T}^2,\gamma)
\longrightarrow C^\infty(\mathbb{T}^2,\gamma\otimes \Omega^1(\mathbb{T}^2)).
$$
Then we define the corresponding noncommutative connection
$\nabla_P=P\nabla_\ga P$ and its curvature form
$$
F_P=(\nabla_P)^2=\left(P  d  P+\frac {2\pi i}\theta d\psi P\varphi P\right)^2.
$$
Straightforward computation shows that the curvature form is the operator of
multiplication by the  noncommutative   2-form
$$
F_P=PdPdP+\frac {2\pi i}\theta d\psi \bigl(-Pd\varphi +P[[\varphi,P],dP]\bigr).
$$
Since the coefficients of  $P$ do not depend on $\psi$, we have $dP=d\varphi
\frac {dP}{d\varphi}$ and thus
$$
F_P=\frac {2\pi i}\theta d\varphi d\psi \Bigl(P-P\left[[\varphi,P],\frac
{dP}{d\varphi}\right]\Bigr).
$$
The index formula \eqref{eq-dirac-index} in the special case of  $\mathbb{T}^2$
gives us
$$
\ind D=-\int_{\mathbb{T}^2}A(T^*\mathbb{T}^2)\ch_e(P).
$$
(There is minus sign, because we deal with the adjoint of the Dirac operator on
$\mathbb{T}^2$.) The cotangent bundle of the torus is trivial, thus
$A(T^*\mathbb{T}^2)=1$ and, therefore, the index is determined by the degree
two component of the Chern character:
$$
\ind D=\frac 1{2\pi i}\int_{\mathbb{T}^2} \tau_e(F_P).
$$
Substituting the curvature form in this equation, we obtain
\begin{multline}\label{ind-r}
\ind D=\frac 1{2\pi i}\int_{\mathbb{T}^2} \tau_e(F_P)=\frac 1{2\pi i}\int_{\mathbb{T}^2} F_P(0)=\\
=\frac 1{2\pi i}
   \int_{
          \begin{array}{c}
            0\le \varphi\le \theta, \\
            0\le \psi\le 1 \\
          \end{array}
         } F_P(0)
=\frac 1\theta \int_{0\le \varphi\le \theta}
  \Bigl(P-P\left[[\varphi,P],\frac {dP}{d\varphi}\right]\Bigr) (0)
  d\varphi.
\end{multline}
Here the second equality follows from the fact the the integral over
$\mathbb{T}^2$ already contains averaging over $\varphi$, therefore, additional
averaging, which is contained in $\tau_e$, is superfluous, and can be
  omitted.

Let now  $P$ be the Rieffel projection \cite{Rie1}:
$$
P=U^{-1}g+f+gU,
$$
where $f,g$ are smooth functions with period $\theta$, which are defined as
follows. Let us choose  $\varepsilon>0$ small enough, so that the intervals
$[0,\varepsilon]$ and $[1,1+\varepsilon]$ of the circle $\mathbb{R}/\theta
\mathbb{Z}$ are disjoint. Then we set
$$
f(x)=\left\{%
\begin{array}{cl}
  1, & \text{if } x\in [\varepsilon,1],\\
  0, & \text{if }x>1+\varepsilon, \\
  1-f(1+x),& \text{if } (0,\varepsilon),\\
\end{array}%
\right.
$$
$$
g(x)=
\left\{%
\begin{array}{cl}
  \sqrt{f(x)-f^2(x)}, & \text{if } x\in [0,\varepsilon],\\
  0, & \text{if }x \notin [0,\varepsilon].\\
\end{array}%
\right.
$$
Since $[\varphi,U]=-U$, $[\varphi,U^{-1}]=U^{-1}$, we  get
$$
[\varphi,P]=U^{-1}g-gU.
$$
for the Rieffel projection. Let us substitute this expression in
$P[[\varphi,P],P']$, multiply the forms and retain only terms containing only
zero power of $U$. We get
$$
-U^{-1}g^2Uf'+U^{-1}gf'gU+2fU^{-1}gg'U-2fg'g+g^2f'-gUf'U^{-1}g.
$$
Since we have to apply trace to this expression, we can multiply the first
three summands in the formula by   $U$ on the left and by $U^{-1}$ on the
right. The resulting element is equal to
$$
-g^2Uf'U^{-1}+gf'g+2UfU^{-1}gg'-2fg'g+g^2f'-gUf'U^{-1}g.
$$
Because $UfU^{-1}$ is equal to $1-f$ on the support of $g$, the last expression
is equal to
\begin{multline*}
f'g^2+f'g^2+2g'g(1-2f)+f'g^2+g^2f'=4g^2f'+(1-2f)(g^2)'=\\
f'(4f-4f^2+(1-2f)^2)=f'.
\end{multline*}
Let us substitute this formula in~\eqref{ind-r} and obtain finally
\begin{equation}\label{index-number}
\ind D=\frac 1\theta\left(\int_0^\theta f(\varphi)d\varphi-
\bigl.f(\varphi)\bigr|_0^\varepsilon\right)= \left\{\frac 1
\theta\right\}-\frac 1 \theta=-\left[\frac 1 \theta\right],
\end{equation}
where $\{\cdot\}$ and $[\cdot]$ stand for the fractional and integer part of a
real number.

It follows from \eqref{index-number} that for $\theta$ sufficiently small
operator \eqref{eq-oper-connes} has large cokernel.


\vfill

{\it Peoples' Friendship University of Russia, Moscow

Leibnitz Universit\"at Hannover}

\end{document}